\documentclass[11pt,reqno]{article}
\usepackage{mathrsfs}
\usepackage{hyperref}
\usepackage{amsmath,amssymb,amsthm}
\newtheorem{theorem}{Theorem}[section]
\newtheorem{lemma}[theorem]{Lemma}

\newtheorem{corollary}[theorem]{Corollary}

\newtheorem{remark}[theorem]{Remark}
\usepackage[english]{babel}
\numberwithin{equation}{section}

\title{\bf Local well-posedness for the inhomogeneous biharmonic nonlinear Schr\"{o}dinger equation in Sobolev spaces}
\author{{JinMyong An, PyongJo Ryu, JinMyong Kim$^*$}\\
\footnotesize{Faculty of Mathematics, {\bf Kim Il Sung} University, Pyongyang, Democratic People's Republic of Korea}\\
\footnotesize{$^*$ Corresponding Author: JinMyong Kim (jm.kim0211@ryongnamsan.edu.kp)}}
\date{}
\begin{document}
\maketitle
\begin{abstract}
In this paper, we study the Cauchy problem for the inhomogeneous biharmonic nonlinear Schr\"{o}dinger (IBNLS) equation
\[iu_{t} +\Delta^{2} u=\lambda |x|^{-b}|u|^{\sigma}u,~u(0)=u_{0} \in H^{s} (\mathbb R^{d}),\]
where $d\in \mathbb N$, $s\ge 0$, $0<b<4$, $\sigma>0$ and $\lambda \in \mathbb R$.
Under some regularity assumption for the nonlinear term, we prove that the IBNLS equation is locally well-posed in $H^{s}(\mathbb R^{d})$ if $d\in \mathbb N$, $0\le s <\min \{2+\frac{d}{2},\frac{3}{2}d\}$, $0<b<\min\{4,d,\frac{3}{2}d-s,\frac{d}{2}+2-s\}$ and $0<\sigma< \sigma_{c}(s)$. Here $\sigma_{c}(s)=\frac{8-2b}{d-2s}$ if $s<\frac{d}{2}$, and $\sigma_{c}(s)=\infty$ if $s\ge \frac{d}{2}$. Our local well-posedness result improves the ones of Guzm\'{a}n-Pastor [Nonlinear Anal. Real World Appl. 56 (2020) 103174] and Liu-Zhang [J. Differential Equations 296 (2021) 335-368] by extending the validity of $s$ and $b$.
\end{abstract}



\section{Introduction}\label{sec 1.}

In this paper, we study the Cauchy problem for the inhomogeneous biharmonic nonlinear Schr\"{o}dinger (IBNLS) equation
\begin{equation} \label{GrindEQ__1_1_}
\left\{\begin{array}{l} {iu_{t} +\Delta^{2}u=\lambda |x|^{-b} |u|^{\sigma} u,~(t,x)\in \mathbb R\times \mathbb R^{d},}\\
{u(0,x)=u_{0}(x) \in H^{s}(\mathbb R^{d})}, \end{array}\right.
\end{equation}
where $d\in \mathbb N$, $s\ge 0$, $0<b<4$, $\sigma>0$ and $\lambda \in \mathbb R$.
The equation \eqref{GrindEQ__1_1_} may be seen as an inhomogeneous version of the classic biharmonic nonlinear Schr\"{o}dinger equation (also called fourth-order NLS equation),
\begin{equation} \label{GrindEQ__1_2_}
iu_{t} +\Delta^{2}u=\lambda |u|^{\sigma} u.
\end{equation}
The biharmonic nonlinear Schr\"{o}dinger equation \eqref{GrindEQ__1_2_} was introduced by Karpman \cite{K96} and Karpman-Shagalov \cite{KS97} to take into account the role of small fourth-order dispersion terms in the propagation of intense laser beams in a bulk medium with Kerr nonlinearity and it has attracted a lot of interest during the last two decades. See, for example, \cite{D18B,D18I,D21,GC07,LZ21,MZ16,P07,PX13} and the references therein.
The local and global well-posedness as well as scattering and blow-up in the energy space $H^{2}$ have been widely studied.
See \cite{D21,MZ16,P07,PX13} for example.
On the other hand, the local and global well-posedness in the fractional Sobolev spaces $H^s$ for the 4NLS equation \eqref{GrindEQ__1_1_} have also been studied by several authors. See \cite{D18B,D18I,GC07,LZ21} for example.

The equation \eqref{GrindEQ__1_1_} has a counterpart for the Laplacian operator, namely, the inhomogeneous nonlinear Schr\"{o}dinger (INLS) equation
\begin{equation} \label{GrindEQ__1_3_}
iu_{t} +\Delta u=\lambda |x|^{-b} |u|^{\sigma} u.
\end{equation}
The INLS equation \eqref{GrindEQ__1_3_} arises in nonlinear optics for modeling the propagation of laser beam and it has been widely studied by many authors. See, for example, \cite{AT21,AK211,AK212,AKC21,C21,DK21,GM21, MMZ21} and the references therein.
The local and global well-posedness as well as blow-up and scattering in the energy space $H^{1}$ for \eqref{GrindEQ__1_3_} have been widely studied by many authors.
See, for example, \cite{C21,DK21,GM21, MMZ21} and the references therein.
Meanwhile, the local and global well-posedness for the INLS equation \eqref{GrindEQ__1_3_} in the fractional Sobolev space have also been studied in \cite{AT21,AK211,AK212,AKC21}.

In this paper, we are interested in studying the well-posedness for the IBNLS equation \eqref{GrindEQ__1_1_} in Sobolev spaces $H^{s}(\mathbb R^{d})$.
Before recalling the known results for \eqref{GrindEQ__1_1_}, let us give some facts about this equation.
The IBNLS equation \eqref{GrindEQ__1_1_} has the conserved mass and energy, defined respectively by
\begin{equation} \label{GrindEQ__1_4_}
M\left(u(t)\right):=\int_{\mathbb R^{d}}{| u(t,x)|^{2}dx}=M\left(u_{0} \right),
\end{equation}
\begin{equation} \label{GrindEQ__1_5_}
E\left(u(t)\right):=\frac{1}{2}\int_{\mathbb R^{d}}{|\Delta u(t,x)|^2 dx}-\frac{\lambda
}{\sigma+2}\int_{\mathbb R^{d}}{|x|^{-b}\left|u(t,x)\right|^{\sigma+2}dx}=E\left(u_{0} \right).
\end{equation}
The IBNLS equation \eqref{GrindEQ__1_1_} is invariant under scaling $u_{\alpha}(t,x)=\alpha^{\frac{4-b}{\sigma}}u(\alpha^{4}t,\alpha x ),~\alpha >0$.
An easy computation shows that
$$
\left\|u_{\alpha}(t)\right\|_{\dot{H}^{s}}=\alpha^{s+\frac{4-b}{\sigma}-\frac{d}{2}}\left\|u(t)\right\|_{\dot{H}^{s}}.
$$
We thus define the critical Sobolev index
\begin{equation} \label{GrindEQ__1_6_}
s_{c}:=\frac{d}{2}-\frac{4-b}{\sigma}.
\end{equation}
Putting
\begin{equation} \label{GrindEQ__1_7_}
\sigma_{c}(s):=
\left\{\begin{array}{cl}
{\frac{8-2b}{d-2s},} ~&{{\rm if}~s<\frac{d}{2},}\\
{\infty,}~&{{\rm if}~s\ge \frac{d}{2},}
\end{array}\right.
\end{equation}
we can easily see that $s>s_{c}$ is equivalent to $\sigma<\sigma_{c}(s)$. If $s<\frac{d}{2}$, $s=s_{c}$ is equivalent to $\sigma=\sigma_{c}(s)$.
For initial data $u_{0}\in H^{s}(\mathbb R^{d})$, we say that the Cauchy problem \eqref{GrindEQ__1_1_} is $H^{s}$-critical (for short, critical) if $0\le s<\frac{d}{2}$ and $\sigma=\sigma_{c}(s)$.
If $s\ge 0$ and $\sigma<\sigma_{c}(s)$, the problem \eqref{GrindEQ__1_1_} is said to be  $H^{s}$-subcritical (for short, subcritical).
Especially, if $\sigma =\frac{8}{d-2s}$, the problem is known as $L^{2}$-critical or mass-critical.
If $\sigma =\frac{8-2b}{d-4}$ with $d\ge 5$, it is called $H^{2}$-critical or energy-critical.
Throughout the paper, a pair $(p,q)$ is said to be  admissible, for short $(p,q)\in A$, if
\begin{equation}\label{GrindEQ__1_8_}
\frac{2}{p}+\frac{d}{q}\le\frac{d}{2},~p\in [2,\infty],~q\in [2,\infty).
\end{equation}
Especially, we say that a pair $(p,q)$ is biharmonic Schr\"{o}dinger admissible (for short $(p,q)\in B$), if $(p,q)\in A$ with $\frac{4}{p}+\frac{d}{q}=\frac{d}{2}$. We also say that a pair $(p,q)$ is non-endpoint biharmonic Schr\"{o}dinger admissible (for short $(p,q)\in B_{0}$), if $(p,q)\in B$ with $p>2$. Note that $(p,q)\in B_{0}$ if, and only if, $\frac{4}{p}+\frac{d}{q}=\frac{d}{2}$ and
\begin{equation} \nonumber
\left\{\begin{array}{ll}
{2\le q< \frac{2d}{d-2}},~&{{\rm if}~d\ge 5,} \\
{2\le q<\infty} ,~&{{\rm if}~d\le 4.}
\end{array}\right.
\end{equation}
If $(p,q)\in A$ and $\frac{2}{p}+\frac{d}{q}=\frac{d}{2}$, then a pair $(p,q)$ is said to be Schr\"{o}dinger admissible (for short $(p,q)\in S$).
We also denote for $(p,q)\in [1,\infty]^{2}$,
\begin{equation}\label{GrindEQ__1_9_}
\gamma_{p,q}=\frac{d}{2}-\frac{4}{p}-\frac{d}{q}.
\end{equation}

The IBNLS equation \eqref{GrindEQ__1_1_} has attracted a lot of interest in recent years. See, for example, \cite{CG21, CGP20, GP20, GP21, LZ212, S21} and the references therein.
Guzm\'{a}n-Pastor \cite{GP20} proved that \eqref{GrindEQ__1_1_} is locally well-posed in $L^{2}$, if $d\in \mathbb N$, $0<b<\min\left\{4,d\right\}$ and $0<\sigma<\sigma_{c}(0)$. They also established the local well-posedness in $H^{2}$ for $d\ge 3$, $0<b<\min\{4,\frac{d}{2}\}$, $\max\{0,\frac{2-2b}{d}\}<\sigma<\sigma_{c}(2)$.
Afterwards, Cardoso-Guzm\'{a}n-Pastor \cite{CGP20} established the local and global well-posedness in $\dot{H}^{s_{c}}\cap \dot{H}^{2}$ with $d\ge 5$, $0<s_{c}<2$, $0<b<\min\{4,\frac{d}{2}\}$ and $\max\{1,\frac{8-2b}{d}\}<\sigma< \frac{8-2b}{d-4}$.
Recently, Liu-Zhang \cite{LZ212} established the local well-posedness in $H^{s}$ with $0<s<2$ by using the Besov space theory.
More precisely, they proved that the IBNLS equation \eqref{GrindEQ__1_1_} is locally well-posed in $H^{s}$ if $d\in \mathbb N$, $0<s\le 2$, $0<\sigma<\sigma_{c}(s)$, $0<b<\min\{4,\frac{d}{2}\}$. See Theorem 1.5 of \cite{LZ212} for details.
This result about the local well-posedness of \eqref{GrindEQ__1_1_} improves the one of \cite{GP20} by not only extending the validity of $d$ and $s$ but also removing the lower bound $\sigma>\frac{2-2b}{d}$.
This local well-posedness result is directly applied to obtain the global well-posedness result in $H^{2}$ for the mass-subcritical case $0<\sigma< \frac{8-2b}{d}$ and mass-critical case  $\sigma=\frac{8-2b}{d}$. See Corollary 1.10 of \cite{LZ212}. The global well-posedness and scattering in $H^{2}$ for the intercritical case $\frac{8-2b}{d}<\sigma<\sigma_{c}(2)$ have also been widely studied. See \cite{CG21, CGP20, GP20, GP21, S21} for example.

In this paper, we investigate the local well-posedness in the fractional Sobolev spaces $H^s$ with $0\le s <\min \{2+\frac{d}{2},\frac{3}{2}d\}$ for the IBNLS equation \eqref{GrindEQ__1_1_}.
To arrive at this goal, we establish the various delicate nonlinear estimates and use the contraction mapping principle combined with the Strichartz estimates.
Our main result is the following.
\begin{theorem}\label{thm 1.1.}
Let $d\in \mathbb N$, $0\le s <\min \{2+\frac{d}{2},\frac{3}{2}d\}$, $0<b<\min\{4,d,\frac{3}{2}d-s,\frac{d}{2}+2-s\}$  and $0<\sigma<\sigma_{c}(s)$.
If $\sigma$ is not an even integer, assume that
\footnote[1]{Given $a\in \mathbb R$, $a^{+}$ denotes the fixed number slightly larger than $a$ ($a^{+}=a+\varepsilon$ with $\varepsilon>0$ small enough).}
\begin{equation} \label{GrindEQ__1_10_}
\sigma\ge \sigma_{\star}(s):=
\left\{\begin{array}{ll}
{0,}~&{{\rm if}~{d\in \mathbb N ~{\rm and}~s\le 1,}}\\
{\left[ s-\frac{d}{2}\right]},~&{{\rm if}~d=1,2 ~{\rm and}~ s\ge 1+\frac{d}{2}},\\
{\lceil s\rceil-2},~&{{\rm if}~d\ge 3~{\rm and}~s>2},\\
{\left(\frac{2s-2b-2}{d}\right)^{+}},~&{\rm otherwise.}
\end{array}\right.
\end{equation}
 Then for any $u_{0} \in H^{s}$, there exist $T_{\max}=T_{\max } \left(\left\|u_{0} \right\| _{H^{s} } \right)>0$, $T_{\min}=T_{\min } \left(\left\|u_{0} \right\| _{H^{s} } \right)>0$ and a unique, maximal solution of \eqref{GrindEQ__1_1_} satisfying
\begin{equation} \label{GrindEQ__1_11_}
u\in C\left(\left(-T_{\min},T_{\max} \right),H^{s} \right)\cap L_{\rm loc}^{p} \left(\left(-T_{\min },T_{\max} \right),H_{q}^{s} \right),
\end{equation}
for any $(p,q)\in B$. If $s>0$, then the solution of \eqref{GrindEQ__1_1_} depends continuously on the initial data $u_{0} $ in the following sense. There exists $0<T<T_{\max } ,\, T_{\min } $ such that if $u_{0,n} \to u_{0} $ in $H^{s}$ and if $u_{n} $ denotes the solution of \eqref{GrindEQ__1_1_} with the initial data $u_{0,n}$, then $0<T<T_{\max } \left(u_{0,n} \right),T_{\min } \left(u_{0,n}\right)$ for all sufficiently large $n$ and $u_{n} \to u$ in $L^{p} \left([-T,T],H_{q}^{s-\varepsilon } \right)$ as $n\to \infty $ for all $\varepsilon >0$ and all $(p,q)\in B$. In particular, $u_{n} \to u$ in $C\left([-T,T],H^{s-\varepsilon } \right)$ for all $\varepsilon >0$.
\end{theorem}

\begin{remark}\label{rem 1.2.}
\textnormal{Theorem \ref{thm 1.1.} improves the local well-posedness result of \cite{GP20} by extending the validity of $d$ and $s$.}
\end{remark}
Let us compare Theorem \ref{thm 1.1.} with the best known local well-posedness result of \cite{LZ212}.

First, Liu-Zhang \cite{LZ212} only considered the case $s\le 2$. As an immediate consequence of Theorem \ref{thm 1.1.}, we have the following new local well-posedness result in $H^{s}(\mathbb R^{d})$ with $d\ge 2$ and $2<s<2+\frac{d}{2}$.
\begin{corollary}\label{cor 1.3.}
Let $d\ge 2$, $2<s<2+\frac{d}{2}$, $0<b<\min\{4,\frac{d}{2}+2-s\}$ and $0<\sigma<\sigma_{c}(s)$.
If $\sigma$ is not an even integer, assume further that
\begin{equation} \label{GrindEQ__1_12_}
\left\{\begin{array}{ll}
{\sigma\ge\left[ s-\frac{d}{2}\right]},~&{{\rm if}~d=2,}\\
{\sigma\ge\lceil s\rceil-2},~&{{\rm if}~d\ge 3.}\\
\end{array}\right.
\end{equation}
Then the IBNLS equation \eqref{GrindEQ__1_1_} is locally well-posed in $H^{s}$.
\end{corollary}

When $0<s<2$, the local well-posedness result of \cite{LZ212} only covers the case $b<\min\{4,\frac{d}{2}\}$. As an immediate consequence of Theorem \ref{thm 1.1.}, we also have the following new result, which covers the case $\frac{d}{2}\le b< 4$ with $d<8$.

\begin{corollary}\label{cor 1.4.}
Let $d<8$, $0<s<\min \{2,d\}$, $\frac{d}{2}\le b<\min\{4,d,\frac{3}{2}d-s,\frac{d}{2}+2-s\}$ and $0<\sigma<\sigma_{c}(s)$.
Then the IBNLS equation \eqref{GrindEQ__1_1_} is locally well-posed in $H^{s}$.
\end{corollary}
\begin{remark}\label{rem 1.5.}
\textnormal{In Corollary \ref{cor 1.4.}, the assumption \eqref{GrindEQ__1_10_} becomes trivial, since $\frac{2s-2b-2}{d}<0$.}
\end{remark}

Combining Corollaries \ref{cor 1.3.} and \ref{cor 1.4.} directly with the local well-posedness result of \cite{LZ212}, we obtain the following local well-posedness results for the IBNLS equation \eqref{GrindEQ__1_1_}.
\begin{corollary}\label{cor 1.6.}
Let $d\ge 2$, $0\le s<2+\frac{d}{2}$, $0<b<\min\{4,d,\frac{d}{2}+2-s\}$ and $0<\sigma<\sigma_{c}(s)$.
If $\sigma$ is not an even integer, assume further that
\begin{equation} \label{GrindEQ__1_13_}
\left\{\begin{array}{ll}
{\sigma\ge\left[ s-\frac{d}{2}\right]},~&{{\rm if}~d=2~{\rm and}~s>2,}\\
{\sigma\ge\lceil s\rceil-2},~&{{\rm if}~d\ge 3~{\rm and}~s>2.}\\
\end{array}\right.
\end{equation}
Then the IBNLS equation \eqref{GrindEQ__1_1_} is locally well-posed in $H^{s}$.
\end{corollary}

\begin{corollary}\label{cor 1.7.}
Let $d=1$, $0\le s\le 2$, $0<b<\tilde{4}$ and $0<\sigma<\sigma_{c}(s)$, where
\begin{equation} \label{GrindEQ__1_14_}
\tilde{4}=
\left\{\begin{array}{ll}
{\min\{1,\frac{3}{2}-s\}},~&{{\rm if}~s<1,}\\
{\frac{1}{2}},~&{{\rm if}~s\ge 1.}\\
\end{array}\right.
\end{equation}
Then the IBNLS equation \eqref{GrindEQ__1_1_} is locally well-posed in $H^{s}$.
\end{corollary}

This paper is organized as follows. In Section \ref{sec 2.}, we recall some notation and give some preliminary results related to our problem. In Section \ref{sec 3.}, we establish the various nonlinear estimates and prove Theorem \ref{thm 1.1.}.

\section{Preliminaries}\label{sec 2.}
Let us introduce some notation used throughout the paper.
$C>0$ stands for a positive universal constant, which can be different at different places.
The notation $a\lesssim b$ means $a\le Cb$ for some constant $C>0$.
For $p\in \left[1,\infty \right]$, $p'$ denotes the dual number of $p$, i.e. $1/p+1/p'=1$.
For $s\in\mathbb R$, we denote by $[s]$ the largest integer which is less than or equal to $s$ and by $\left\lceil s\right\rceil $ the minimal integer which is larger than or equal to $s$.
As in \cite{WHHG11}, for $s\in \mathbb R$ and $1<p<\infty $, we denote by $H_{p}^{s} (\mathbb R^{d} )$ and $\dot{H}_{p}^{s} (\mathbb R^{d} )$ the nonhomogeneous Sobolev space and homogeneous Sobolev space, respectively. As usual, we abbreviate $H_{2}^{s} (\mathbb R^{d} )$ and $\dot{H}_{2}^{s} (\mathbb R^{d} )$ as $H^{s} (\mathbb R^{d} )$ and $\dot{H}^{s} (\mathbb R^{d} )$, respectively. Given two normed spaces $X$ and $Y$, $X\hookrightarrow Y$ means that $X$ is continuously embedded in $Y$, i.e. there exists a constant $C\left(>0\right)$ such that $\left\| f\right\| _{Y} \le C\left\| f\right\| _{X} $ for all $f\in X$. If there is no confusion, $\mathbb R^{d} $ will be omitted in various function spaces.

Next, we recall some useful facts which are used throughout the paper.
First of all, we recall some useful embeddings on Sobolev spaces. See \cite{WHHG11} for example.

\begin{lemma}\label{lem 2.1.}
Let $-\infty <s_{2} \le s_{1} <\infty $ and $1<p_{1} \le p_{2} <\infty $ with $s_{1} -\frac{d}{p_{1} } =s_{2} -\frac{d}{p_{2} } $. Then we have the following embeddings:
\[\dot{H}_{p_{1} }^{s_{1} } \hookrightarrow \dot{H}_{p_{2} }^{s_{2} } ,H_{p_{1} }^{s_{1} } \hookrightarrow H_{p_{2} }^{s_{2} } .\]
\end{lemma}

\begin{lemma}\label{lem 2.2.}
Let $-\infty <s<\infty $ and $1<p<\infty $. Then we have
\begin{enumerate}
\item $H_{p}^{s+\varepsilon } \hookrightarrow H_{p}^{s}~(\varepsilon >0)$,
\item $H_{p}^{s} =L^{p} \cap \dot{H}_{p}^{s}~(s>0)$.
\end{enumerate}
\end{lemma}

\begin{corollary}\label{cor 2.3.}
Let $-\infty <s_{2} \le s_{1} <\infty $ and $1<p_{1} \le p_{2} <\infty $ with $s_{1} -\frac{d}{p_{1} } \ge s_{2} -\frac{d}{p_{2} } $. Then we have $H_{p_{1} }^{s_{1} } \hookrightarrow H_{p_{2} }^{s_{2} } $.
\end{corollary}
\begin{proof} The result follows from Lemma \ref{lem 2.1.} and Item 1 of Lemma \ref{lem 2.2.}.
\end{proof}

\begin{lemma}[\cite{K95}]\label{lem 2.4.}
Let $F\in C^{k}(\mathbb C,\mathbb C)$, $k\in \mathbb N \setminus \{0\}$. Assume that there is $\nu\ge k$ such that
$$
|D^{i}F(z)|\lesssim |z|^{\nu-i},~z\in \mathbb C,~i=1,2,\ldots,k.
$$
Then for $s \in [0,k]$ and $1<r,p<\infty$, $1<q\le \infty$ satisfying $\frac{1}{r}=\frac{1}{p}+\frac{\nu-1}{q}$, we have
$$
\left\|F(u)\right\|_{\dot{H}_{r}^{s}}\lesssim \left\|u\right\|_{L^q}^{\nu-1}\left\|u\right\|_{\dot{H}_{p}^{s}}.
$$
\end{lemma}

Lemma \ref{lem 2.4.} applies in particular to the model case $F(u)=\lambda|u|^{\sigma}u$ with $\lambda\in \mathbb R$.
\begin{corollary}\label{cor 2.5.}
Let $\sigma>0$, $s\ge 0$ and $1<r,p<\infty$, $1<q\le \infty$ satisfying $\frac{1}{r}=\frac{1}{p}+\frac{\sigma}{q}$. If $\sigma$ is not an even integer, assume that $\sigma\ge \lceil s\rceil-1$. Then we have
$$
\left\||u|^{\sigma}u\right\|_{\dot{H}_{r}^{s}}\lesssim \left\|u\right\|_{L^q}^{\sigma}\left\|u\right\|_{\dot{H}_{p}^{s}}.
$$
\end{corollary}

Next, we recall the well-known fractional product rule.

\begin{lemma}[Fractional Product Rule, \cite{CW91}]\label{lem 2.6.}
Let $s\ge 0$, $1<r,r_{2},p_{1}<\infty$ and $1<r_{1},p_{2}\le\infty$. Assume that
\[\frac{1}{r} =\frac{1}{r_{i} } +\frac{1}{p_{i} }~(i=1,2).\]
Then we have
\begin{equation}
\left\| fg\right\| _{\dot{H}_{r}^{s} } \lesssim \left\| f\right\| _{r_{1} } \left\| g\right\| _{\dot{H}_{p_{1} }^{s} } +\left\| f\right\| _{\dot{H}_{r_{2} }^{s} } \left\| g\right\| _{p_{2} } .
\end{equation}
\end{lemma}

We end this section with recalling the Strichartz estimates for the fourth-order Schr\"{o}dinger equation. See, for example, \cite{D18I} and the reference therein.
\begin{lemma}[Strichartz estimates]\label{lem 2.7.}
Let $\gamma\in \mathbb R$ and $u$ be the solution to the linear fourth-order Schr\"{o}dinger equation, namely
\begin{equation}\nonumber
u(t)=e^{it\Delta^{2}}u_{0}+i\int_{0}^{t}{e^{i(t-s)\Delta^{2}}F(s)ds,}
\end{equation}
for some data $u_{0}$ and $F$. Then for all $(p,q)$ and $(a,b)$ admissible,
$$
\left\|u\right\|_{L^{p}(I,\dot{H}_{q}^{\gamma})}\lesssim  \left\|u_{0}\right\|_{\dot{H}^{\gamma+\gamma_{p,q}}}+ \left\|F\right\|_{L^{a'}(I,\dot{H}_{b'}^{\gamma+\gamma_{p,q}-\gamma_{a',b'}-4})},
$$
where $\gamma_{p,q}$ and $\gamma_{a',b'}$ are as in \eqref{GrindEQ__1_9_}.
\end{lemma}
\section{Proof of main result}\label{sec 3.}

In this section, we prove Theorem \ref{thm 1.1.}. Before establishing the various nonlinear estimates, we recall the following useful fact. See Remark 3.4 of \cite{AK212}.
\begin{remark}\label{rem 3.1.}
\textnormal{Let $1<r<\infty$, $s\ge0$ and $b>0$. Let $\chi\in C_{0}^{\infty}(\mathbb R^{d})$ satisfy $\chi(x)=1$ for $|x|\le 1$ and $\chi(x)=0$ for $|x|\ge 2$.
If $b+s<\frac{d}{r}$, then $\chi(x)|x|^{-b} \in \dot{H}_{r}^{s}$.
If $b+s>\frac{d}{r}$, then $(1-\chi(x))|x|^{-b}\in \dot{H}_{r}^{s}$.}
\end{remark}
\begin{lemma}\label{lem 3.2.}
$d\in \mathbb N$, $0\le s <\min \{2+\frac{d}{2},\frac{3}{2}d\}$, $0<b<\min\{4,d,\frac{3}{2}d-s,\frac{d}{2}+2-s\}$  and $0<\sigma<\sigma_{c}(s)$. If $\sigma$ is not an even integer, assume \eqref{GrindEQ__1_10_}. Then there exist $(a_{i},b_{i})\in A$ $(i=1,2)$, $(p_{j},q_{j})\in B_{0}$ $(j=1,2,3,4)$ and $\theta>0$ such that
\begin{equation} \label{GrindEQ__3_1_}
\left\|\chi(x)|x|^{-b}|u|^{\sigma}u\right\|_{L^{a_{1}'}(I,\dot{H}_{b_{1}'}^{s-\gamma_{a_{1}',b_{1}'}-4})}
\lesssim |I|^{\theta}\left\|u\right\|_{L^{p_{1}}(I,H_{q_{1}}^{s})}^{\sigma+1},
\end{equation}
\begin{equation} \label{GrindEQ__3_2_}
\left\|(1-\chi(x))|x|^{-b}|u|^{\sigma}u\right\|_{L^{a_{2}'}(I,\dot{H}_{b_{2}'}^{s-\gamma_{a_{2}',b_{2}'}-4})}
\lesssim |I|^{\theta}\max_{j\in\{2,3,4\}}\left\|u\right\|_{L^{p_{j}}(I,H_{q_{j}}^{s})}^{\sigma+1},
\end{equation}
where $I(\subset \mathbb R)$ is an interval and $\chi\in C_{0}^{\infty}(\mathbb R^{d})$ is given in Remark \ref{rem 3.1.}.
\end{lemma}
\begin{proof} If $(a_{i},b_{i})\in A$ $(i=1,2)$, then we can see that
\begin{equation} \label{GrindEQ__3_3_}
\left\{\begin{array}{l}
{s-\gamma_{a_{i}',b_{i}'}-4\ge 0 \Leftrightarrow \frac{4}{a_{i}}\le s+\frac{d}{2}-\frac{d}{b_{i}},}\\
{\gamma_{a_{i}',b_{i}'}+4\ge 0 \Leftrightarrow \frac{4}{a_{i}}\ge \frac{d}{2}-\frac{d}{b_{i}}.}
\end{array}\right.
\end{equation}
We divide the study in two cases: $0\le s<\frac{d}{2}$ and $s>\frac{d}{2}$.

{\bf Case 1.} We consider the case $0\le s<\frac{d}{2}$.

First, we prove \eqref{GrindEQ__3_1_}.
Let $(a_{1},b_{1})\in A$ and $(p_{1},q_{1})\in B_{0}$ satisfy the following system:
\begin{equation} \label{GrindEQ__3_4_}
\left\{\begin{array}{l}
{\frac{1}{b_{1}'}-\sigma\left(\frac{1}{q_{1}}-\frac{s}{d}\right)-\left(\frac{1}{q_{1}}-\frac{\gamma_{a_{1}',b_{1}'}+4}{d}\right)>\frac{b}{d},}\\
{\frac{1}{a_{1}'}-\frac{\sigma +1}{p_{1}}>0,}\\
{\frac{1}{q_{1}}-\frac{s}{d}>0,}\\
{s-\gamma_{a_{1}',b_{1}'}-4\ge0,}\\
{\gamma_{a_{1}',b_{1}'}+4\ge 0.}\\
\end{array}\right.
\end{equation}
Putting
\begin{equation} \label{GrindEQ__3_5_}
\left\{\begin{array}{l}
{\frac{1}{\alpha_{1}}:=\frac{1}{q_{1}}-\frac{s}{d},~\frac{1}{\beta_{1}}:=\frac{1}{q_{1}}-\frac{\gamma_{a_{1}',b_{1}'}+4}{d},}\\
{\frac{1}{r_{1}}:=\frac{\sigma}{\alpha_{1}}+\frac{1}{\beta_{1}},~\frac{1}{\gamma _{1}}:=\frac{1}{b_{1}'}-\frac{1}{r_{1}},}\\
{\frac{1}{\bar{r}_{1}}:=\frac{\sigma}{\alpha_{1}},~~\frac{1}{\bar{\gamma}_{1}}:=\frac{1}{b_{1}'}-\frac{1}{\bar{r}_{1}},} \\
\end{array}\right.
\end{equation}
it follows from Lemma \ref{lem 2.1.}, the third and last equations in \eqref{GrindEQ__3_4_} that
$$\dot{H}_{q_{1} }^{s}\hookrightarrow L^{\alpha_{1}}, ~\dot{H}_{q_{1} }^{s}\hookrightarrow\dot{H}_{\beta_{1} }^{s-\gamma_{a_{1}',b_{1}'}-4}.$$
Furthermore, it follows from Remark \ref{rem 3.1.} and the first equation in \eqref{GrindEQ__3_4_} that
$$
\left\| \chi(x)|x|^{-b} \right\| _{L^{\gamma _{1}}}<\infty, ~\left\| \chi(x)|x|^{-b} \right\| _{\dot{H}_{\bar{\gamma}_{1} }^{s-\gamma_{a_{1}',b_{1}'}-4} }<\infty.
$$
Hence, using the fourth equation in \eqref{GrindEQ__3_4_}, \eqref{GrindEQ__3_5_}, Lemma \ref{lem 2.6.} and Corollary \ref{cor 2.5.}, we have
\begin{eqnarray}\begin{split} \label{GrindEQ__3_6_}
&\left\| \chi(x)|x|^{-b} |u|^{\sigma}u\right\| _{\dot{H}_{b_{1}'}^{s-\gamma_{a_{1}',b_{1}'}-4}}\\
&~~~~~~~\lesssim \left\| \chi(x)|x|^{-b} \right\| _{L^{\gamma _{1}}} \left\| |u|^{\sigma}u\right\| _{\dot{H}_{r_{1} }^{s-\gamma_{a_{1}',b_{1}'}-4} }
          +\left\| \chi(x)|x|^{-b} \right\| _{\dot{H}_{\bar{\gamma}_{1} }^{s-\gamma_{a_{1}',b_{1}'}-4} } \left\| |u|^{\sigma}u\right\| _{L^{\bar{r}_{1} } }\\
&~~~~~~~\lesssim \left\| u\right\|^{\sigma} _{L^{\alpha_{1}}}\left\|u\right\| _{\dot{H}_{\beta_{1} }^{s-\gamma_{a_{1}',b_{1}'}-4} }
          +\left\| u\right\|^{\sigma+1} _{L^{\alpha_{1}}}
        \lesssim \left\|u\right\|^{\sigma+1}_{\dot{H}_{q_{1} }^{s}},
\end{split}\end{eqnarray}
where we need the assumption that if $\sigma$ is not an even integer,
\begin{equation} \label{GrindEQ__3_7_}
\sigma\ge \lceil{ s-\gamma_{a_{1}',b_{1}'}-4}\rceil-1.
\end{equation}
Using the H\"{o}lder inequality, \eqref{GrindEQ__3_6_} and the second equation in the system \eqref{GrindEQ__3_4_}, we immediately have \eqref{GrindEQ__3_1_}.
Therefore, it remains to choose $(a_{1},b_{1})\in A$ and $(p_{1},q_{1})\in B_{0}$ satisfying \eqref{GrindEQ__3_4_}.
In fact, we can easily see that there exists $(p_{1},q_{1})\in B_{0}$ satisfying \eqref{GrindEQ__3_4_}, if
\begin{equation} \label{GrindEQ__3_8_}
\max\left\{\frac{d-4}{2d},\frac{s}{d},\frac{1}{2}-\frac{4}{da_{1}'(\sigma+1)}\right\}<\frac{1}{q_{1}}
<\min\left\{\frac{1}{2},\frac{s}{d}+\frac{1}{\sigma+1}\left(\frac{1}{2}-\frac{b+s}{d}+\frac{4}{a_{1}d}\right)\right\}.
\end{equation}
The equation \eqref{GrindEQ__3_8_} holds, if $s<\frac{d}{2}$, $\sigma<\sigma_{c}(s)$ and
\begin{equation} \label{GrindEQ__3_9_}
\frac{4}{a_{1}}>\max\left\{b+s-\frac{d}{2},b-2+\frac{(d-2s-4)\sigma}{2}\right\}.
\end{equation}
In view of \eqref{GrindEQ__1_8_}, \eqref{GrindEQ__3_3_} and \eqref{GrindEQ__3_9_}, it suffices to choose $a_{1}\in [2,\infty]$ and $b_{1}\in [2,\infty)$ satisfying
\begin{equation} \label{GrindEQ__3_10_}
\max\left\{b+s-\frac{d}{2},\frac{d}{2}-\frac{d}{b_{1}},b-2+\frac{(d-2s-4)\sigma}{2}\right\}<\frac{4}{a_{1}}\le \min\left\{d-\frac{2d}{b_{1}},s+\frac{d}{2}-\frac{d}{b_{1}}\right\}.
\end{equation}

If $1<s<\frac{d}{2}$ with $d\ge 3$, then we put $\frac{4}{a_{1}}:=2$ and $\frac{d}{b_{1}}:=\frac{d}{2}-1$. Using the fact $b<2+\frac{d}{2}-s$, we can see that $(a_{1}, b_{1})\in A$ satisfies \eqref{GrindEQ__3_10_}.
And the assumption \eqref{GrindEQ__3_7_} becomes
\begin{equation} \label{GrindEQ__3_11_}
\sigma\ge \lceil{ s-\gamma_{a_{1}',b_{1}'}-4}\rceil-1=\lceil{s\rceil}-2.
\end{equation}

If $0\le s\le 1$ with $s<\frac{d}{2}$, then we have
\begin{equation} \label{GrindEQ__3_12_}
s+\frac{d}{2}-\frac{d}{b_{1}}\le \min\left\{2,d-\frac{2d}{b_{1}}\right\}\Leftrightarrow \frac{d}{b_{1}}\in \left[s-2+\frac{d}{2},\frac{d-2s}{2}\right].
\end{equation}
Hence, putting $\frac{4}{a_{1}}:=s+\frac{d}{2}-\frac{d}{b_{1}}$, it suffices to choose $b_{1}$ satisfying
\begin{equation} \label{GrindEQ__3_13_}
\left\{\begin{array}{l}
{s-2+\frac{d}{2}\le \frac{d}{b_{1}}\le \frac{d-2s}{2},}\\
{0<\frac{d}{b_{1}}<\min\left\{d-b,\frac{d}{2},\frac{2s+d-2b+4}{2}-\frac{(d-2s-4)\sigma}{2}\right\}.}\\
\end{array}\right.
\end{equation}
Using the fact $b<\min\left\{d, \frac{d}{2}+2-s\right\}$ and $\sigma<\sigma_{c}(s)$, we can easily choose $b_{1}$ satisfying \eqref{GrindEQ__3_13_}. And the assumption \eqref{GrindEQ__3_7_} becomes
\begin{equation} \label{GrindEQ__3_14_}
\sigma\ge \left\lceil s-\gamma_{a_{1}',b_{1}'}-4\right\rceil-1=-1,
\end{equation}
which is trivial. This completes the proof of \eqref{GrindEQ__3_1_}.

Next, we prove \eqref{GrindEQ__3_2_}.
Let $(a_{2},b_{2})\in A$ and $(p_{i},q_{i})\in B_{0}$ $(i=2,3,4)$ satisfy the following system:
\begin{equation} \label{GrindEQ__3_15_}
\left\{\begin{array}{ll}
{\frac{1}{q_{2}}-\frac{s}{d}\le\frac{1}{\alpha_{2}}\le\frac{1}{q_{2}},\frac{1}{q_{3}}-\frac{\gamma_{a_{2}',b_{2}'}+4}{d}\le\frac{1}{\beta_{2}}\le\frac{1}{q_{3}}},~&~~~~~~~~{(1)}\\
{0<\frac{1}{b_{2}'}-\frac{\sigma}{\alpha_{2}}-\frac{1}{\beta_{2}}<\frac{b}{d},}~&~~~~~~~~{(2)}\\
{\frac{1}{a_{2}'}-\frac{\sigma}{p_{2}}-\frac{1}{p_{3}}>0,}~&~~~~~~~~{(3)}\\
{\frac{1}{q_{2}}>\frac{s}{d},\frac{1}{q_{3}}>\frac{\gamma_{a_{2}',b_{2}'}+4}{d},}~&~~~~~~~~{(4)}\\
{\frac{1}{q_{4}}-\frac{s}{d}\le\frac{1}{\alpha_{4}}\le\frac{1}{q_{4}},}~&~~~~~~~~{(5)}\\
{0<\frac{1}{b_{2}'}-\frac{\sigma+1}{\alpha_{4}}<\frac{b+s-\gamma_{a_{2}',b_{2}'}-4}{d},}~&~~~~~~~~{(6)}\\
{\frac{1}{a_{2}'}-\frac{\sigma+1}{p_{4}}>0,\frac{1}{q_{4}}>\frac{s}{d},}~&~~~~~~~~{(7)}\\
{s-\gamma_{a_{2}',b_{2}'}-4\ge0,\gamma_{a_{2}',b_{2}'}+4\ge 0.}~&~~~~~~~~{(8)}\\
\end{array}\right.
\end{equation}
Using Lemma \ref{lem 2.2.}, Corollary \ref{cor 2.3.} and the equations (1), (4), (5), (7), (8) in the system \eqref{GrindEQ__3_15_}, we have the embeddings:
$$
H_{q_{2}}^{s}\hookrightarrow L^{\alpha_{2}},~H_{q_{3}}^{s}\hookrightarrow \dot{H}_{\beta_{2}}^{s-\gamma_{a_{2}',b_{2}'}-4},
~H_{q_{4} }^{s}\hookrightarrow L^{\alpha_{4}}.
$$
Putting
\begin{equation} \label{GrindEQ__3_16_}
\left\{\begin{array}{l}
{\frac{1}{r_{2}}:=\frac{\sigma}{\alpha_{2}}+\frac{1}{\beta_{2}},~\frac{1}{\gamma _{2}}:=\frac{1}{b_{2}'}-\frac{1}{r_{2}},}\\
{\frac{1}{r_{4}}:=\frac{\sigma+1}{\alpha_{4}},~\frac{1}{\gamma _{4}}:=\frac{1}{b_{2}'}-\frac{1}{r_{4}},}\\
\end{array}\right.
\end{equation}
it follows from Remark \ref{rem 3.1.} and the equations (2), (6) in \eqref{GrindEQ__3_15_} that
$$
\left\|(1-\chi(x))|x|^{-b} \right\| _{L^{\gamma _{2}}}<\infty,
~\left\| (1-\chi(x))|x|^{-b}\right\| _{\dot{H}_{\gamma_{4}}^{s-\gamma_{a_{2}',b_{2}'}-4}}<\infty.
$$
Hence, using \eqref{GrindEQ__3_16_}, Corollary \ref{cor 2.5.} and the equation (8) in the system \eqref{GrindEQ__3_15_}, we have
\begin{eqnarray}\begin{split} \label{GrindEQ__3_17_}
&\left\| (1-\chi(x))|x|^{-b} |u|^{\sigma}u\right\| _{\dot{H}_{b_{2}'}^{s-\gamma_{a_{2}',b_{2}'}-4}}\\
&~~~~~\lesssim \left\|(1-\chi(x))|x|^{-b} \right\| _{L^{\gamma _{2}}} \left\| |u|^{\sigma}u\right\|_{\dot{H}_{r_{2}}^{s-\gamma_{a_{2}',b_{2}'}-4} }
          +\left\| (1-\chi(x))|x|^{-b}\right\| _{\dot{H}_{\gamma_{4}}^{s-\gamma_{a_{2}',b_{2}'}-4}} \left\||u|^{\sigma}u\right\|_{L^{r_{4}}}\\
&~~~~~\lesssim \left\|u\right\|^{\sigma} _{L^{\alpha_{2}}}\left\|u\right\| _{\dot{H}_{\beta_{2}}^{s-\gamma_{a_{2}',b_{2}'}-4} }
          +\left\|u\right\|^{\sigma+1} _{L^{\alpha_{4}}}
        \lesssim \left\|u\right\|^{\sigma}_{H_{q_{2}}^{s}}\left\|u\right\|_{H_{q_{3}}^{s}}+\left\|u\right\|^{\sigma+1}_{H_{q_{4} }^{s}},
\end{split}\end{eqnarray}
where we need the assumption that if $\sigma$ is not an even integer,
\begin{equation} \label{GrindEQ__3_18_}
\sigma\ge \left\lceil{ s-\gamma_{a_{2}',b_{2}'}-4}\right\rceil-1.
\end{equation}
Then, \eqref{GrindEQ__3_2_} follows directly from the H\"{o}lder inequality, \eqref{GrindEQ__3_16_} and the equations (3), (7) in the system \eqref{GrindEQ__3_15_}.
Therefore, it remains to choose $(a_{2},b_{2})\in A$ and $(p_{i},q_{i})\in B_{0}$ $(i=2,3,4)$ satisfying \eqref{GrindEQ__3_15_}.
In fact, it is easy to see that we can choose $\alpha_{2}$ and $\beta_{2}$ satisfying the equations (1) and (2) in the system \eqref{GrindEQ__3_15_}, if
\begin{equation} \label{GrindEQ__3_19_}
1-\frac{1}{b_{2}}-\frac{b}{d}-\frac{\sigma}{q_{2}}<\frac{1}{q_{3}}<\frac{1}{2}+\frac{\sigma s}{d}+\frac{4}{a_{2}d}-\frac{\sigma}{q_{2}}.
\end{equation}
In view of (3) and (4) in the system \eqref{GrindEQ__3_15_}, we can choose ($p_{3}, q_{3})\in B_{0}$ satisfying \eqref{GrindEQ__3_15_}, if
\begin{equation} \label{GrindEQ__3_20_}
\left\{\begin{array}{l}
{\max\left\{1-\frac{1}{b_{2}}-\frac{b}{d}-\frac{\sigma}{q_{2}},\frac{\sigma+1}{2}-\frac{4}{d}+\frac{4}{a_{2}d}-\frac{\sigma}{q_{2}},\frac{d-4}{2d},
\;-\frac{1}{2}+\frac{4}{a_{2}d}+\frac{1}{b_{2}}\right\}<\frac{1}{q_{3}}},\\
{\frac{1}{q_{3}}<\min\left\{\frac{1}{2},\frac{1}{2}+\frac{\sigma s}{d}+\frac{4}{a_{2}d}-\frac{\sigma}{q_{2}}\right\}.}
\end{array}\right.
\end{equation}
Using the fact $\sigma<\sigma_{c}(s)$, we can easily see that \eqref{GrindEQ__3_20_} holds if
\begin{equation} \label{GrindEQ__3_21_}
\frac{d}{2}-\frac{d}{b_{2}}-b-\sigma s<\frac{4}{a_{2}}<d-\frac{d}{b_{2}},
\end{equation}
\begin{equation} \label{GrindEQ__3_22_}
\max\left\{\frac{1}{2}-\frac{1}{b_{2}}-\frac{b}{d},\frac{\sigma}{2}-\frac{4}{d}+\frac{4}{a_{2}d}\right\}<\frac{\sigma}{q_{2}}
<\min\left\{\frac{\sigma s}{d}+\frac{4}{a_{2}d}+\frac{2}{d},1+\frac{\sigma s}{d}-\frac{1}{b_{2}}\right\}.
\end{equation}
In view of \eqref{GrindEQ__3_22_} and the equation (4) in \eqref{GrindEQ__3_15_}, there exists $(p_{2},q_{2})\in B_{0}$ satisfying \eqref{GrindEQ__3_15_}, if
\begin{equation} \label{GrindEQ__3_23_}
\left\{\begin{array}{l}
{\max\left\{\frac{\sigma s}{d},\frac{(d-4)\sigma}{2d},\frac{1}{2}-\frac{1}{b_{2}}-\frac{b}{d},\frac{\sigma}{2}-\frac{4}{d}+\frac{4}{a_{2}d}\right\}<\frac{\sigma}{q_{2}}},\\
{\frac{\sigma}{q_{2}}<\min\left\{\frac{\sigma}{2},\frac{\sigma s}{d}+\frac{4}{a_{2}d}+\frac{2}{d},1+\frac{\sigma s}{d}-\frac{1}{b_{2}}\right\}.}
\end{array}\right.
\end{equation}
One can easily see that \eqref{GrindEQ__3_23_} holds, if
\begin{equation} \label{GrindEQ__3_24_}
\left\{\begin{array}{l}
{\frac{d}{2}-b-\frac{d\sigma}{2}<\frac{d}{b_{2}}<d-\frac{(d-2s-4)\sigma}{2},}\\
{\max\left\{\frac{d}{2}-\frac{d}{b_{2}}-b-\sigma s-2,\frac{(d-2s-4)\sigma}{2}-2 \right\}<\frac{4}{a_{2}}<d-\frac{(d-2s)\sigma}{2}+4-\frac{d}{b_{2}}.}\\
\end{array}\right.
\end{equation}
On the other hand, an easy computation shows that there exist $\alpha_{4}$ satisfying the equations (5) and (6) in the system \eqref{GrindEQ__3_15_}, if
\begin{equation} \label{GrindEQ__3_25_}
\frac{1}{2}-\frac{b+s}{d}+\frac{4}{a_{2}d}<\frac{\sigma+1}{q_{4}}<1-\frac{1}{b_{2}}+\frac{(\sigma+1)s}{d}.
\end{equation}
And the equation (7) in the system \eqref{GrindEQ__3_15_} holds, if
\begin{equation} \label{GrindEQ__3_26_}
\frac{\sigma+1}{q_{4}}>\max\left\{\frac{\sigma+1}{2}-\frac{4}{da_{2}'},\frac{(\sigma+1)s}{d}\right\}.
\end{equation}
In view of \eqref{GrindEQ__3_25_} and \eqref{GrindEQ__3_26_}, it suffices to choose $q_{4}$ satisfying
\begin{equation} \nonumber
\left\{\begin{array}{l}
{\max\left\{\frac{1}{2}-\frac{b+s}{d}-\frac{4}{a_{2}d},\frac{\sigma+1}{2}-\frac{4}{da_{2}'},\frac{(\sigma+1)s}{d},\frac{(d-4)(\sigma+1)}{2d}\right\}<\frac{\sigma+1}{q_{4}},}\\
{\frac{\sigma+1}{q_{4}}<\min\left\{\frac{\sigma+1}{2},1-\frac{1}{b_{2}}+\frac{(\sigma+1)s}{d}\right\},}
\end{array}\right.
\end{equation}
which is possible, if
\begin{equation} \label{GrindEQ__3_27_}
\frac{d}{b_{2}}<d-\frac{(\sigma+1)(d-2s-4)}{2},
\end{equation}
\begin{equation} \label{GrindEQ__3_28_}
\frac{4}{a_{2}}<\min\left\{\frac{d\sigma}{2}+b+s,\frac{d}{2}+b+(\sigma+2)s-\frac{d}{b_{2}},\frac{d}{2}+(\sigma+1)s+4-\frac{d\sigma}{2}-\frac{d}{b_{2}}\right\}.
\end{equation}
Noticing $(a_{2},b_{2})\in A$, $s<\frac{d}{2}$ and $\sigma<\sigma_{c}(s)$ and using \eqref{GrindEQ__1_8_}, \eqref{GrindEQ__3_3_}, \eqref{GrindEQ__3_21_}, \eqref{GrindEQ__3_24_}, \eqref{GrindEQ__3_27_}, \eqref{GrindEQ__3_28_}, it suffices to choose $a_{2},b_{2}$ satisfying
\begin{equation} \label{GrindEQ__3_29_}
\left\{\begin{array}{l}
{\max\left\{0,\frac{d}{2}-b-\frac{d\sigma}{2}\right\}<\frac{d}{b_{2}}<\min\left\{\frac{d}{2},d-\frac{(\sigma+1)(d-2s-4)}{2}\right\},}\\
{\frac{(d-2s-4)\sigma}{2}-2<\frac{4}{a_{2}}<\frac{d\sigma}{2}+b+s,}\\
{\frac{d}{2}-\frac{d}{b_{2}}\le\frac{4}{a_{2}}\le \min\left\{2,d-\frac{2d}{b_{2}},s+\frac{d}{2}-\frac{d}{b_{2}}\right\}.}
\end{array}\right.
\end{equation}
A computation shows that there exist $a_{2}$, $b_{2}$ satisfying \eqref{GrindEQ__3_29_} and we omit the details.

Our purpose is to choose $a_{2}$ and $b_{2}$ satisfying \eqref{GrindEQ__3_29_} so that the assumption \eqref{GrindEQ__3_18_} becomes the best.
If $s\le1$, the assumption \eqref{GrindEQ__3_18_} is trivial, due to the equation (8) in \eqref{GrindEQ__3_15_}. Therefore, it suffices to consider the case $1<s<\frac{d}{2}$ with $d\ge 3$.

If $\sigma>\frac{2-2b}{d}$, we can see that $\left(a_{2},b_{2}\right):=\left(2, \frac{2d}{d-2}\right)$ satisfies \eqref{GrindEQ__3_29_}. And the assumption \eqref{GrindEQ__3_18_} becomes
\begin{equation} \label{GrindEQ__3_30_}
\sigma\ge \left\lceil{ s-\gamma_{a_{2}',b_{2}'}-4}\right\rceil-1=\left\lceil{s}\right\rceil-2.
\end{equation}

If $\sigma\le \frac{2-2b}{d}$, we can see that
$\frac{4}{a_{2}}:=d-\frac{2d}{b_{2}}$ and $\frac{d}{b_{2}}:=\left(b+\frac{(d-2s)\sigma}{2}\right)^{+}$ satisfy \eqref{GrindEQ__3_29_}. And the assumption \eqref{GrindEQ__3_18_} becomes
\begin{equation} \label{GrindEQ__3_31_}
\sigma\ge \left\lceil{ s-\gamma_{a_{2}',b_{2}'}-4}\right\rceil-1=\left[s-b-\frac{d\sigma}{2}\right].
\end{equation}

Let us analyze the assumptions \eqref{GrindEQ__3_30_} and \eqref{GrindEQ__3_31_}.

First, we consider the case $s\le2$. If $\sigma>\frac{2-2b}{d}$, then the assumption \eqref{GrindEQ__3_30_} is trivial. If $\frac{2-2b}{d}\ge \sigma>\frac{2s-2-2b}{d}$ with $s>2$, then the assumption \eqref{GrindEQ__3_31_} is trivial. If $\sigma\le \frac{2s-2-2b}{d}$, then \eqref{GrindEQ__3_31_} implies that $\sigma\ge 1$, which is impossible.
Therefore, the lower bound on $\sigma$ for $s\le 2$ becomes $\sigma>\frac{2s-2-2b}{d}$.

Next, we consider the case $s>2$. If $\sigma\le \frac{2-2b}{d}$, then the assumption \eqref{GrindEQ__3_31_} implies that $\sigma\ge 1$, which contracts $\sigma\le \frac{2-2b}{d}$.
Noticing $\left\lceil{s}\right\rceil-2\ge 1>\frac{2-2b}{d}$, the lower bound on $\sigma$ in the case $s>2$ becomes $\sigma\ge\left\lceil{s}\right\rceil-2$.
This completes the proof in the case $s<\frac{d}{2}$.

{\bf Case 2.} We consider the case $s>\frac{d}{2}$.

First, we prove \eqref{GrindEQ__3_1_}. Let $(a_{1},b_{1})\in S$, i.e., we have
$$
\frac{2}{a_{1}}+\frac{d}{b_{1}}=\frac{d}{2},~s-\gamma_{a_{1}',b_{1}'}-4=s-\frac{2}{a_{1}}.
$$
For $(a_{1},b_{1})\in S$ satisfying
\begin{equation} \label{GrindEQ__3_32_}
\left\{\begin{array}{l}
{2\le\alpha_{1},\alpha_{2}<\infty,~2\le \beta_{1}\le b_{1}},\\
{\frac{1}{b_{1}'}-\frac{\sigma}{\alpha_{1}}-\frac{1}{\beta_{1}}>\frac{b}{d},~\frac{1}{b_{1}'}-\frac{\sigma}{\alpha_{2}}>\frac{b+s-\frac{2}{a_{1}}}{d}},\\
\end{array}\right.
\end{equation}
it follows from Lemma \ref{lem 2.6.}, Corollary \ref{cor 2.5.}, Corollary \ref{cor 2.3.} and Remark \ref{rem 3.1.} that
\begin{eqnarray}\begin{split} \label{GrindEQ__3_33_}
&\left\| \chi(x)|x|^{-b} |u|^{\sigma}u\right\| _{\dot{H}_{b_{1}'}^{s-\frac{2}{a_{1}}}}\\
&~~~~~~~~~\lesssim \left\|\chi(x)|x|^{-b} \right\| _{L^{\gamma _{1}}} \left\| |u|^{\sigma}u\right\|_{\dot{H}_{r_{1}}^{s-\frac{2}{a_{1}}}}
          +\left\| \chi(x)|x|^{-b}\right\| _{\dot{H}_{\bar{\gamma}_{1}}^{s-\frac{2}{a_{1}}}} \left\||u|^{\sigma}u\right\|_{L^{\bar{r}_{1}}}\\
&~~~~~~~~~\lesssim \left\|u\right\|^{\sigma} _{L^{\alpha_{1}}}\left\|u\right\| _{\dot{H}_{\beta_{1}}^{s-\frac{2}{a_{1}}}}
          +\left\|u\right\|^{\sigma+1} _{L^{\bar{\alpha}_{1}}}
        \lesssim \left\|u\right\|^{\sigma+1}_{H^{s}},
\end{split}\end{eqnarray}
where
\begin{equation} \label{GrindEQ__3_34_}
\left\{\begin{array}{l}
{\frac{1}{r_{1}}:=\frac{\sigma}{\alpha_{1}}+\frac{1}{\beta_{1}},~\frac{1}{\gamma _{1}}:=\frac{1}{b_{1}'}-\frac{1}{r_{1}},}\\
{\frac{1}{\bar{r}_{1}}:=\frac{\sigma+1}{\bar{\alpha}_{1}},~\frac{1}{\bar{\gamma}_{1}}:=\frac{1}{b_{1}'}-\frac{1}{\bar{r}_{1}}.}\\
\end{array}\right.
\end{equation}
Here, we need the assumption that if $\sigma$ is not an even integer,
\begin{equation} \label{GrindEQ__3_35_}
\sigma\ge \left\lceil s-\frac{2}{a_{1}}\right\rceil-1.
\end{equation}
Using \eqref{GrindEQ__3_33_} and H\"{o}lder inequality, we immediately get \eqref{GrindEQ__3_1_} with $q_{1}=2$.
Therefore, it remains to choose $(a_{1},b_{1})\in S$ satisfying \eqref{GrindEQ__3_32_} so that the assumption \eqref{GrindEQ__3_35_} becomes the best.
It is easy to see that we can choose $\alpha_{1}$, $\bar{\alpha}_{1}$ and $\beta_{1}$ satisfying \eqref{GrindEQ__3_32_}, if
\begin{equation} \label{GrindEQ__3_36_}
\left\{\begin{array}{l}
{\frac{1}{b_{1}}\le\frac{\sigma}{\alpha_{1}}+\frac{1}{\beta_{1}}\le \frac{\sigma+1}{2},}\\
{\frac{\sigma}{\alpha_{1}}+\frac{1}{\beta_{1}}<1-\frac{1}{b_{1}}-\frac{b}{d}},
\end{array}\right.
\end{equation}
and
\begin{equation} \label{GrindEQ__3_37_}
\left\{\begin{array}{l}
{0<\frac{\sigma+1}{\bar{\alpha}_{1}}\le \frac{\sigma+1}{2},}\\
{\frac{\sigma+1}{\bar{\alpha}_{1}}<\frac{3}{2}-\frac{b+s}{d}-\frac{2}{b_{1}}.}
\end{array}\right.
\end{equation}
Since $s\ge \frac{d}{2}$, we can see that \eqref{GrindEQ__3_36_} and \eqref{GrindEQ__3_37_} hold, if $\frac{2}{b_{1}}<\frac{3}{2}-\frac{b+s}{d}$.
If $d\ge 3$, we can take $b_{1}:=\frac{2d}{d-2}$ using the fact $b<2+\frac{d}{2}-s$.
Here, the assumption \eqref{GrindEQ__3_35_} becomes $\sigma\ge \lceil s\rceil-2$. \
If $d\le 2$, we can take $b_{1}(>2)$ large enough such that $\frac{2}{b_{1}}<\frac{3}{2}-\frac{b+s}{d}$ and $\left\lceil s-\frac{d}{2}+\frac{2}{b{1}}\right\rceil=\left[s-\frac{d}{2}\right]$.
This completes the proof of \eqref{GrindEQ__3_1_}.

Next, we prove \eqref{GrindEQ__3_2_}.
For $(a_{2},b_{2})\in A$ and $(p_{i},q_{i})\in B_{0}$ $(i=2,3,4)$ satisfying
\begin{equation} \label{GrindEQ__3_38_}
\left\{\begin{array}{ll}
{0<\frac{1}{\alpha_{2}}\le\frac{1}{q_{2}},\frac{1}{q_{3}}-\frac{\gamma_{a_{2}',b_{2}'}+4}{d}\le\frac{1}{\beta_{2}}\le\frac{1}{q_{3}}},~&~~~~~~~~{(1)}\\
{0<\frac{1}{b_{2}'}-\frac{\sigma}{\alpha_{2}}-\frac{1}{\beta_{2}}<\frac{b}{d},}~&~~~~~~~~{(2)}\\
{\frac{1}{a_{2}'}-\frac{\sigma}{p_{2}}-\frac{1}{p_{3}}>0,}~&~~~~~~~~{(3)}\\
{\frac{1}{q_{3}}>\frac{\gamma_{a_{2}',b_{2}'}+4}{d},}~&~~~~~~~~{(4)}\\
{0<\frac{1}{\alpha_{4}}\le\frac{1}{q_{4}},}~&~~~~~~~~{(5)}\\
{0<\frac{1}{b_{2}'}-\frac{\sigma+1}{\alpha_{4}}<\frac{b+s-\gamma_{a_{2}',b_{2}'}-4}{d},}~&~~~~~~~~{(6)}\\
{\frac{1}{a_{2}'}-\frac{\sigma+1}{p_{4}}>0,}~&~~~~~~~~{(7)}\\
{s-\gamma_{a_{2}',b_{2}'}-4\ge0,\gamma_{a_{2}',b_{2}'}+4\ge 0,}~&~~~~~~~~{(8)}\\
\end{array}\right.
\end{equation}
we get \eqref{GrindEQ__3_16_} and \eqref{GrindEQ__3_17_} by using the similar argument as in the proof of \eqref{GrindEQ__3_2_} in Case 1 and we omit the details. Here, we need to assume \eqref{GrindEQ__3_18_}, if $\sigma$ is not an even integer.
Then, \eqref{GrindEQ__3_2_} follows directly from the H\"{o}lder inequality, \eqref{GrindEQ__3_16_} and the equations (3), (7) in the system \eqref{GrindEQ__3_38_}.
Therefore, it remains to choose $(a_{2},b_{2})\in A$ and $(p_{i},q_{i})\in B_{0}$ $(i=2,3,4)$ satisfying \eqref{GrindEQ__3_38_} so that the condition \eqref{GrindEQ__3_18_} becomes the best.
In fact, it is easy to see that we can choose $\alpha_{2}$ and $\beta_{2}$ satisfying the equations (1) and (2) in the system \eqref{GrindEQ__3_38_}, if
\begin{equation} \label{GrindEQ__3_39_}
1-\frac{1}{b_{2}}-\frac{b}{d}<\frac{\sigma}{q_{2}}+\frac{1}{q_{3}},~\frac{1}{q_{3}}<\frac{1}{2}+\frac{4}{a_{2}d}.
\end{equation}
In view of (3) and (4) in the system \eqref{GrindEQ__3_38_}, we can choose ($p_{i}, q_{i})\in B_{0}$ $(i=1,2)$ satisfying \eqref{GrindEQ__3_38_}, if $q_{2}$ and $q_{3}$ satisfy
\begin{equation} \label{GrindEQ__3_40_}
-\frac{1}{2}+\frac{4}{a_{2}d}+\frac{1}{b_{2}}<\frac{1}{q_{3}}<\frac{1}{2},
\end{equation}
\begin{equation} \label{GrindEQ__3_41_}
\max\left\{1-\frac{1}{b_{2}}-\frac{b}{d},\frac{\sigma+1}{2}-\frac{4}{d}+\frac{4}{a_{2}d}\right\}<\frac{\sigma}{q_{2}}+\frac{1}{q_{3}}<\frac{\sigma+1}{2}.
\end{equation}
And \eqref{GrindEQ__3_40_} and \eqref{GrindEQ__3_41_} hold, if
\begin{equation} \label{GrindEQ__3_42_}
\frac{d}{b_{2}}>\frac{d}{2}-b-\frac{d\sigma}{2},~\frac{4}{a_{2}}<d-\frac{d}{b_{2}}.
\end{equation}
On the other hand, an easy computation shows that there exist $\alpha_{4}$ satisfying the equations (5) and (6) in the system \eqref{GrindEQ__3_38_}, if
\begin{equation} \label{GrindEQ__3_43_}
\frac{1}{2}-\frac{b+s}{d}+\frac{4}{a_{2}d}<\frac{\sigma+1}{q_{4}}.
\end{equation}
And the equation (7) in the system \eqref{GrindEQ__3_38_} is equivalent to
\begin{equation} \label{GrindEQ__3_44_}
\frac{\sigma+1}{q_{4}}>\frac{\sigma+1}{2}-\frac{4}{da_{2}'}.
\end{equation}
In view of \eqref{GrindEQ__3_43_} and \eqref{GrindEQ__3_44_}, it suffices to choose $q_{4}$ satisfying
\begin{equation} \label{GrindEQ__3_45_}
\max\left\{\frac{1}{2}-\frac{b+s}{d}-\frac{4}{a_{2}d},\frac{\sigma+1}{2}-\frac{4}{da_{2}'},\frac{(\sigma+1)s}{d},\frac{(d-4)(\sigma+1)}{2d}\right\}
<\frac{\sigma+1}{q_{4}}<\frac{\sigma+1}{2},
\end{equation}
which is possible provided that
\begin{equation} \label{GrindEQ__3_46_}
\frac{4}{a_{2}}<\frac{d\sigma}{2}+b+s.
\end{equation}
Noticing $(a_{2},b_{2})\in A$ and $\sigma<\sigma_{c}(s)$ and using \eqref{GrindEQ__1_8_}, \eqref{GrindEQ__3_3_}, \eqref{GrindEQ__3_42_}, \eqref{GrindEQ__3_46_}, it suffices to choose $a_{2},b_{2}$ satisfying
\begin{equation} \label{GrindEQ__3_47_}
\left\{\begin{array}{l}
{\max\left\{0,\frac{d}{2}-b-\frac{d\sigma}{2}\right\}<\frac{d}{b_{2}}<\frac{d}{2},}\\
{\frac{4}{a_{2}}<\frac{d\sigma}{2}+b+s,}\\
{\frac{d}{2}-\frac{d}{b_{2}}\le\frac{4}{a_{2}}\le \min\left\{2,d-\frac{2d}{b_{2}},s+\frac{d}{2}-\frac{d}{b_{2}}\right\}.}
\end{array}\right.
\end{equation}
A computation shows that there exist $a_{2}$, $b_{2}$ satisfying \eqref{GrindEQ__3_47_} and we omit the details.

Our purpose is to choose $a_{2}$ and $b_{2}$ satisfying \eqref{GrindEQ__3_47_} so that the assumption \eqref{GrindEQ__3_18_} becomes the best.
If $s\le1$, the assumption \eqref{GrindEQ__3_18_} is trivial, due to the equation (8) in \eqref{GrindEQ__3_38_}.
If $s>1$ and $d\ge 3$, we use the same argument as in the proof of \eqref{GrindEQ__3_2_} in Case 1 and we omit the details.
Finally, we consider the case $s>1$ and $d=1,2$.

If $\sigma\le \frac{d-2b}{d}$, we can see that $\frac{4}{a_{2}}:=d-\frac{2d}{b_{2}}$ and $\frac{d}{b_{2}}:=\left(b+\frac{(d-2s)\sigma}{2}\right)^{+}$ satisfy \eqref{GrindEQ__3_47_}.
Here, the assumption \eqref{GrindEQ__3_18_} becomes
\begin{equation} \label{GrindEQ__3_48_}
\sigma\ge \lceil{ s-\gamma_{a_{2}',b_{2}'}-4}\rceil-1=\left[s-b-\frac{d\sigma}{2}\right].
\end{equation}

If $\sigma>\frac{d-2b}{d}$, then we can see that $\frac{4}{a_{2}}:=d-\frac{2d}{b_{2}}$ and $\frac{d}{b_{2}}:=\left(0\right)^{+}$ satisfy \eqref{GrindEQ__3_47_}. Here, the assumption \eqref{GrindEQ__3_18_} becomes
\begin{equation} \label{GrindEQ__3_49_}
\sigma\ge \lceil{ s-\gamma_{a_{2}',b_{2}'}-4}\rceil-1=\left[s-\frac{d}{2}\right].
\end{equation}
Let us analyze the assumptions \eqref{GrindEQ__3_48_} and \eqref{GrindEQ__3_49_}.

First, we consider the case $\frac{d}{2}\le s< 1+\frac{d}{2}$.
If $\sigma>\frac{d-2b}{d}$, then the assumption \eqref{GrindEQ__3_49_} is trivial.
If $\frac{d-2b}{d}\ge \sigma>\frac{2s-2-2b}{d}$, then the assumption \eqref{GrindEQ__3_48_} is trivial.
If $\sigma\le \frac{2s-2-2b}{d}$, then \eqref{GrindEQ__3_48_} implies that $\sigma\ge 1$, which is impossible.
Therefore, the lower bound on $\sigma$ for $\frac{d}{2}\le s<1+\frac{d}{2}$ becomes
\begin{equation} \label{GrindEQ__3_50_}
\sigma>\frac{2s-2-2b}{d}.
\end{equation}

Next, we consider the case $s\ge\frac{d}{2}+1$.
If $\sigma\le \frac{d-2b}{d}$, then the assumption \eqref{GrindEQ__3_48_} implies that $\sigma\ge 1$, which is impossible.
Noticing $\left[s-\frac{d}{2}\right]\ge 1>\frac{d-2b}{d}$, the lower bound on $\sigma$ in the case $s\ge1+\frac{d}{2}$ becomes
\begin{equation} \label{GrindEQ__3_51_}
\sigma\ge\left[s-\frac{d}{2}\right],
\end{equation}
this completes the proof.
\end{proof}

\begin{lemma}\label{lem 3.3.}
Let $d\in \mathbb N$, $s\ge 0$, $0<b<\min\left\{4,d\right\}$ and $0<\sigma<\sigma_{c}(s)$.
Then there exist $(a_{i},b_{i})\in B_{0}$ $(i=3,4)$, $(p_{j},q_{j})\in B_{0}$ $(j=5,6,7)$ and $\theta>0$ such that
\begin{equation} \label{GrindEQ__3_52_}
\left\|\chi(x)|x|^{-b}|u|^{\sigma}v\right\|_{L^{a_{3}'}(I,L^{b_{3}'})}
\lesssim |I|^{\theta}\left\|u\right\|^{\sigma}_{L^{p_{5}}(I,H_{q_{5}}^{s})}\left\|v\right\|_{L^{p_{6}}(I,L^{q_{6}})},
\end{equation}
\begin{equation} \label{GrindEQ__3_53_}
\left\|(1-\chi(x))|x|^{-b}|u|^{\sigma}v\right\|_{L^{a_{4}'}(I,L^{b_{4}'})}
\lesssim |I|^{\theta}\left\|u\right\|^{\sigma}_{L^{p_{7}}(I,H_{q_{7}}^{s})}\left\|v\right\|_{L^{p_{7}}(I,L^{q_{7}})}.
\end{equation}
where $I(\subset \mathbb R)$ is an interval and $\chi\in C_{0}^{\infty}(\mathbb R^{d})$ is given in Remark \ref{rem 3.1.}.
\end{lemma}

\begin{proof}
We divide the proof in two cases: $0\le s<\frac{d}{2}$ and $s\ge \frac{d}{2}$.

{\bf Case 1.} We consider the case $0\le s<\frac{d}{2}$.

First, we prove \eqref{GrindEQ__3_52_}.
For $(a_{3},b_{3})\in B_{0}$ and $(p_{i},q_{i})\in B_{0}$ $(i=5,6)$ satisfying
\begin{equation} \label{GrindEQ__3_54_}
\left\{\begin{array}{ll}
{\frac{1}{q_{5}}-\frac{s}{d}\le\frac{1}{\alpha_{5}}\le\frac{1}{q_{5}};\;\frac{1}{q_{5}}>\frac{s}{d}},\\
{\frac{1}{b_{3}'}-\frac{\sigma}{\alpha_{5}}-\frac{1}{q_{6}}>\frac{b}{d},}\\
{\frac{1}{a_{3}'}-\frac{\sigma}{p_{5}}-\frac{1}{p_{6}}>0,}\\
\end{array}\right.
\end{equation}
it follows from the H\"{o}lder inequality, Remark \ref{rem 3.1.} and Corollary \ref{cor 2.3.} that
\begin{eqnarray}\begin{split} \label{GrindEQ__3_55_}
\left\|\chi(x)|x|^{-b}|u|^{\sigma}v\right\|_{ L^{b_{3}'}}&\lesssim
\left\|\chi(x)|x|^{-b}\right\| _{L^{\hat{\gamma}_{1}}} \left\|u\right\|^{\sigma} _{L^{\alpha_{5}}}\left\|v\right\| _{L^{q_{6}}}
\lesssim \left\|u\right\|^{\sigma}_{H_{q_{5}}^{s}}\left\|v\right\|_{L^{q_{6}}},
\end{split}\end{eqnarray}
where $\frac{1}{\hat{\gamma}_{1}}:=\frac{1}{b_{3}'}-\frac{\sigma}{\alpha_{5}}-\frac{1}{\hat{q_{2}}}$.
Using the third equation in \eqref{GrindEQ__3_54_}, \eqref{GrindEQ__3_55_} and the H\"{o}lder inequality, we immediately get \eqref{GrindEQ__3_52_}.
Hence, it suffices to prove that there exist $(a_{3},b_{3})\in B_{0}$ and $(p_{i},q_{i})\in B_{0}$ $(i=5,6)$ satisfying \eqref{GrindEQ__3_54_}.
In fact, we can take $\alpha_{5}$ satisfying \eqref{GrindEQ__3_54_}, if
\begin{equation} \label{GrindEQ__3_56_}
\frac{1}{q_{6}}<1-\frac{1}{b_{3}}-\frac{b}{d}-\frac{\sigma}{q_{5}}+\frac{\sigma s}{d}.
\end{equation}
And the third equation \eqref{GrindEQ__3_54_} is equivalent to
\begin{equation} \label{GrindEQ__3_57_}
\frac{1}{q_{6}}>\frac{\sigma+2}{2}-\frac{4}{d}-\frac{\sigma}{q_{5}}-\frac{1}{b_{3}}.
\end{equation}
Hence, it suffices to choose $q_{6}$ satisfying
\begin{equation} \label{GrindEQ__3_58_}
\max\left\{0,\frac{d-4}{2d},\frac{\sigma+2}{2}-\frac{4}{d}-\frac{\sigma}{q_{5}}-\frac{1}{b_{3}}\right\}<\frac{1}{q_{6}}
<\min\left\{\frac{1}{2},1-\frac{1}{b_{3}}-\frac{b}{d}-\frac{\sigma}{q_{5}}+\frac{\sigma s}{d}\right\},
\end{equation}
which is possible provided that $\sigma<\sigma_{c}(s)$ and $q_{5}$ satisfies
\begin{equation} \label{GrindEQ__3_59_}
\frac{\sigma+1}{2}-\frac{4}{d}-\frac{1}{b_{3}}<\frac{\sigma}{q_{5}}
<\min\left\{1-\frac{1}{b_{3}}-\frac{b}{d}+\frac{\sigma s}{d},\frac{1}{2}-\frac{1}{b_{3}}-\frac{b}{d}+\frac{\sigma s}{d}+\frac{2}{d}\right\}.
\end{equation}
In view of \eqref{GrindEQ__3_59_} and the first equation in the system \eqref{GrindEQ__3_54_}, it suffices to choose $q_{5}$ and $b_{3}$ satisfying
\begin{equation} \label{GrindEQ__3_60_}
\left\{\begin{array}{l}
{\max\left\{\frac{\sigma s}{d},\frac{(d-4)\sigma}{2d},\frac{\sigma+1}{2}-\frac{4}{d}-\frac{1}{b_{3}}\right\}<\frac{\sigma}{q_{5}},}\\
{\frac{\sigma}{q_{5}}<\min\left\{\frac{\sigma}{2},1-\frac{1}{b_{3}}-\frac{b}{d}+\frac{\sigma s}{d},\frac{1}{2}-\frac{1}{b_{3}}-\frac{b}{d}+\frac{\sigma s}{d}+\frac{2}{d}\right\}}.
\end{array}\right.
\end{equation}
And this is possible, if we choose $b_{3}$ satisfying
\begin{equation}\label{GrindEQ__3_61_}
\frac{1}{b_{3}}<\min\left\{1-\frac{b}{d},1-\frac{b}{d}-\frac{(d-2s-4)\sigma}{2d},\frac{1}{2}+\frac{2-b}{d},\frac{1}{2}+\frac{2-b}{d}-\frac{(d-2s-4)\sigma}{2d}\right\}.
\end{equation}
Using the facts that $\sigma<\sigma_{c}(s)$ and $b<\min\left\{4,d\right\}$, we can easily prove that there exist $(a_{3},b_{3})\in B_{0}$ satisfying \eqref{GrindEQ__3_61_}. This complete the proof of \eqref{GrindEQ__3_52_}.

Next, we prove \eqref{GrindEQ__3_53_}.
Putting
$$
a_{4}=p_{7}:=\frac{8(\sigma+2)}{\sigma(d-2s)},~b_{4}=q_{7}:=\frac{d(\sigma+2)}{d+\sigma s},
$$
we can easily see that $(a_{4}, b_{4})=(p_{7}, q_{7})\in B_{0}$. Furthermore, there hold the following identities:
\begin{equation}\label{GrindEQ__3_62_}
\frac{1}{b_{4}'}=\sigma\left(\frac{1}{q_{7}}-\frac{s}{d}\right)+\frac{1}{q_{7}},~ \frac{1}{a_{4}'}-\frac{\sigma+1}{p_{7}}=1-\frac{\sigma(d-2s)}{8}:=\theta.
\end{equation}
Using \eqref{GrindEQ__3_62_}, Lemma \ref{lem 2.1.} and the H\"{o}lder inequality, we have
\begin{equation} \label{GrindEQ__3_63_}
\left\|(1-\chi(x))|x|^{-b}|u|^{\sigma}v\right\|_{L^{b_{4}'}}
\lesssim \left\|(1-\chi(x))|x|^{-b} \right\| _{L^{\infty}} \left\| |u|^{\sigma}v\right\| _{L^{b_{4}'}}
\lesssim \left\|u\right\|^{\sigma}_{\dot{H}_{q_{7} }^{s}}\left\|v\right\|_{L^{q_{7}}}.
\end{equation}
\eqref{GrindEQ__3_53_} follows directly from \eqref{GrindEQ__3_62_}, \eqref{GrindEQ__3_63_} and the H\"{o}lder inequality.

{\bf Case 2.} Let us consider the case $s\ge \frac{d}{2}$.

First, we prove \eqref{GrindEQ__3_52_}.
If $(a_{3},b_{3})\in B_{0}$ and $(p_{i},q_{i})\in B_{0}$ $(i=5,6)$ satisfy
\begin{equation} \label{GrindEQ__3_64_}
\left\{\begin{array}{ll}
{0<\frac{1}{\alpha_{5}}\le\frac{1}{q_{5}},}\\
{\frac{1}{b_{3}'}-\frac{\sigma}{\alpha_{5}}-\frac{1}{q_{6}}>\frac{b}{d},}\\
{\frac{1}{a_{3}'}-\frac{\sigma}{p_{5}}-\frac{1}{p_{6}}>0,}\\
\end{array}\right.
\end{equation}
then we get \eqref{GrindEQ__3_55_} by using the H\"{o}lder inequality, Remark \ref{rem 3.1.} and Corollary \ref{cor 2.3.}.
Using the third equation in \eqref{GrindEQ__3_64_}, \eqref{GrindEQ__3_55_} and the H\"{o}lder inequality, we immediately get \eqref{GrindEQ__3_52_}. Hence, it suffices to prove that there exist $(a_{3},b_{3})\in B_{0}$ and $(p_{i},q_{i})\in B_{0}$ $(i=5,6)$ satisfying \eqref{GrindEQ__3_64_}.
In fact, we can take $\alpha_{5}$ satisfying \eqref{GrindEQ__3_64_}, if
\begin{equation} \label{GrindEQ__3_65_}
\frac{1}{q_{6}}<1-\frac{1}{b_{3}}-\frac{b}{d}.
\end{equation}
In view of \eqref{GrindEQ__3_65_} and \eqref{GrindEQ__3_57_}, it suffices to choose $q_{6}$ satisfying
\begin{equation} \label{GrindEQ__3_66_}
\max\left\{0,\frac{d-4}{2d},\frac{\sigma+2}{2}-\frac{4}{d}-\frac{\sigma}{q_{5}}-\frac{1}{b_{3}}\right\}<\frac{1}{q_{6}}
<\min\left\{\frac{1}{2},1-\frac{1}{b_{3}}-\frac{b}{d}\right\},
\end{equation}
which is possible, if
\begin{equation} \label{GrindEQ__3_67_}
b_{3}<\min\left\{1-\frac{b}{d},\frac{1}{2}+\frac{2-b}{d}\right\},~
\max\left\{\frac{\sigma}{2}-\frac{4-b}{d},\frac{\sigma+1}{2}-\frac{4}{d}-\frac{1}{b_{3}}\right\}<\frac{\sigma}{q_{5}}<\frac{\sigma}{2}.
\end{equation}
In view of \eqref{GrindEQ__3_67_}, it suffices to choose $b_{3}$ satisfying
\begin{equation} \label{GrindEQ__3_68_}
\max\left\{0,\frac{d-2}{2d}\right\}<b_{3}<\min\left\{1-\frac{b}{d},\frac{1}{2}+\frac{2-b}{d}\right\},
\end{equation}
which is possible, due to the fact $b<\min\left\{4,d\right\}$.
This completes the proof of \eqref{GrindEQ__3_52_}.

Next, we prove \eqref{GrindEQ__3_53_}.
We can choose $(a_{4},b_{4})\in B_{0}$ satisfying $\frac{1}{b_{4}'}\in \left(\frac{1}{2},\frac{\sigma+1}{2}\right)$.
We then choose $\alpha_{6}\in (2,\infty)$ such that $\frac{1}{b_{4}'}=\frac{\sigma}{\alpha_{6}}+\frac{1}{2}$.
Using the H\"{o}lder inequality and the embedding $H^{s}\hookrightarrow L^{\alpha_{6}}$, we have
\begin{eqnarray}\begin{split} \nonumber
\left\|(1-\chi(x))|u|^{\sigma}v\right\|_{L^{a_{4}'}(I, L^{b_{4}'})}
\lesssim |I|^{\theta}\left\|u\right\|^{\sigma}_{L^{\infty}(I, L^{\alpha_{6}})}\left\|v\right\|_{L^{\infty}(I, L^{2})}
\lesssim |I|^{\theta}\left\|u\right\|^{\sigma}_{L^{\infty}(I, H^{s})}\left\|v\right\|_{L^{\infty}(I, L^{2})},
\end{split}\end{eqnarray}
where $\theta:=1-\frac{1}{a_{4}}$. This completes the proof.
\end{proof}

Now we are ready to prove Theorem \ref{thm 1.1.}.
\begin{proof}[{\bf Proof of Theorem \ref{thm 1.1.}}]
Let $T>0$ and $M>0$ which will be chosen later. Given $I=[-T,T]$, we define the complete metric space $(D,d)$ as follows:
\begin{equation} \nonumber
D=\{u\in \cap_{i=1}^{7}L^{p_{i}}(I, H_{q_{i}}^{s}):~\max_{i=\overline{1,7}}\left\|u\right\|_{L^{p_{i}}(I, H_{q_{i}}^{s})} \le M\},
\end{equation}
\begin{equation} \nonumber
d(u,v)=\max_{i=\overline{1,7}}\left\|u-v\right\|_{L^{p_{i}}(I, L^{q_{i}})},
\end{equation}
where $(p_{i},q_{i})\in B_{0}$ are given in Lemma \ref{lem 3.2.} and Lemma \ref{lem 3.3.}.
Now we consider the mapping
\begin{equation}\label{GrindEQ__3_69_}
\mathrm{T}:~u(t)\to e^{it\Delta^{2}}u_{0} -i\lambda \int _{0}^{t}e^{i(t-\tau)\Delta^{2}}|x|^{-b}|u(\tau)|^{\sigma}u(\tau)d\tau .
\end{equation}
Lemma \ref{lem 2.7.} (Strichartz estimates) yields that
\begin{eqnarray}\begin{split} \label{GrindEQ__3_70_}
\max_{i=\overline{1,7}}\left\|\mathrm{T}u\right\|_{L^{p_{i}}(I,H_{q_{i}}^{s})}
&\lesssim \left\| u_{0} \right\| _{H^{s} }+\left\|\chi(x)|x|^{-b}|u|^{\sigma}u\right\|_{L^{a_{1}'}(I,\dot{H}_{b_{1}'}^{s-\gamma_{a_{1}',b_{1}'}-4})}\\
&~+\left\|(1-\chi(x))|x|^{-b}|u|^{\sigma}u\right\|_{L^{a_{2}'}(I,\dot{H}_{b_{2}'}^{s-\gamma_{a_{2}',b_{2}'}-4})}\\
&~+\left\|\chi(x)|x|^{-b}|u|^{\sigma}u\right\|_{L^{a_{3}'}(I,L^{b_{3}'})}\\
&~+\left\|(1-\chi(x))|x|^{-b}|u|^{\sigma}u\right\|_{L^{a_{4}'}(I,L^{b_{4}'})},
\end{split}\end{eqnarray}
where $(a_{i}, b_{i})$ $(i=\overline{1,4})$ are given in Lemma \ref{lem 3.2.} and Lemma \ref{lem 3.3.}.
Using \eqref{GrindEQ__3_70_}, Lemma \ref{lem 3.2.} and Lemma \ref{lem 3.3.}, we have
\begin{equation}\label{GrindEQ__3_71_}
\max_{i=\overline{1,7}}\left\|\mathrm{T}u\right\|_{L^{p_{i}}(I,H_{q_{i}}^{s})} \le C \left\| u_{0} \right\| _{H^{s} }+CT^{\theta}M^{\sigma+1}.
\end{equation}
Similarly, using Lemma \ref{lem 2.7.} (Strichartz estimates) and Lemma \ref{lem 3.3.}, we also have
\begin{eqnarray}\begin{split}\label{GrindEQ__3_72_}
d(\mathrm{T}u,\mathrm{T}v)&\lesssim \left\|\chi(x)|x|^{-b}(|u|^{\sigma }+|v|^{\sigma})|u-v|\right\|_{L^{a_{3}'}(I,L^{b_{3}'})}\\
&~+\left\|(1-\chi(x))|x|^{-b}(|u|^{\sigma }+|v|^{\sigma})|u-v|\right\|_{L^{a_{4}'}(I,L^{b_{4}'})}\\
&\le 2C T^{\theta} M^{\sigma}d(u,v).
\end{split}\end{eqnarray}
Putting $M=2C\left\| u_{0} \right\| _{H^{s} } $ and taking $T>0$ satisfying $
CT^{\theta}M^{\sigma } \le \frac{1}{4}$, we can see that $\mathrm{T}: (D,d)\to (D,d)$ is contraction mapping.  Using the standard argument (see e.g. Section 4.9 in \cite{C03} or Section 4.3 in \cite{WHHG11}), we get the desired results.
\end{proof}


\end{document}